\documentclass[11pt]{article}
\usepackage{amsmath, amssymb, amsfonts, amsthm, latexsym, enumitem,
  myamsdefs, graphics}
\usepackage[sort&compress,square,numbers]{natbib}
\usepackage{algorithm,algorithmic}
\usepackage[active]{srcltx}
\usepackage[colorlinks]{hyperref}
\usepackage{mathrsfs}

\newtheorem{theorem}{Theorem}
\newtheorem{prop}[theorem]{Proposition}
\newtheorem{cor}[theorem]{Corollary}

\newtheorem{example}{Example}
\numberwithin{equation}{section}

\def\Che{characteristic exponent\xspace}
\def\che{\Psi}                   
\def\dgamma{\mathrm{Gamma}}
\def\dgeo{\mathrm{Geo}}
\def\dunif{\mathrm{Unif}}
\def\mean{\mathrm{E}}
\def\pr{\mathrm{P}}

\def\bm#1{B\sb{#1}}
\def\fht#1{\tau_{#1}}   
\def\fet#1{\eta_{#1}}   

\def\jump#1{\Delta#1}
\def\dn{\phi}          
\def\dfh#1{f_{#1}}     
\def\dfe#1{p_{#1}}     
\def\dfex#1{q_{#1}}    
\makeatletter
\def\pdf{p.d.f\@ifnextchar.{}{.}}
\def\pmf{p.m.f\@ifnextchar.{}{.}}
\makeatother

\newlist{enum}{enumerate}{10}
\setlist[enum]{itemsep=0ex, leftmargin=1.5\parindent,
  topsep=.2\baselineskip, parsep=.2\baselineskip, label={\arabic*}.}

\def\ot{\leftarrow}
\begin{document}
\begin{center}
  \large
  \textbf{
    Exact sampling of first passage event of certain symmetric \levy
    processes with unbounded variation
  }
  \\[1.5ex]
  \normalsize
  Zhiyi Chi \\
  Department of Statistics\\
  University of Connecticut \\
  Storrs, CT 06269, USA, \\[.5ex]
  E-mail: zhiyi.chi@uconn.edu \\[1ex]
  \today
\end{center}

\begin{abstract}
  We show that exact sampling of the first passage event can be done
  for a \levy process with unbounded variation, if the process can be
  embedded in a subordinated standard Brownian motion.  By sampling a
  series of first exit events of the Brownian motion and first passage
  events of the subordinator, the first passage event of interest can
  be obtained.  The sampling of the first exit time and pre-exit
  location of the Brownian motion may be of independent interest.

  \medbreak\noindent
  
  \medbreak\noindent
\end{abstract}

\section{Introduction} \label{s:intro}
The first exit event of a \levy process is an intensively studied
subject in probability \cite {bertoin:96:cup, bertoin:99:spl,
  sato:99:cup, doney:02:ap, doney:06:aap, huzak:04:jap,
  kluppelberg:04:aap, alili:05:aap}.  Despite numerous deep
theoretical results on the subject, exact sampling of the first
passage event remains challenging especially for processes with
infinite \levy measures.  Recently, in \cite{chi:16:spa,
  chi:12:advap}, it was shown that exact sampling can be done for a
wide range of \levy processes with bounded variation.  The methods of
\cite {chi:12:advap, chi:16:spa} rely on the decomposition of a
process as the  difference of two independent subordinators.  However,
no such decomposition exists for \levy processes with unbounded
variation.  In this paper, we show that exact sampling of the first
exit event can be achieved for several important classes of such
processes, a primary example being those that have a symmetric
truncated $\alpha$-stable \levy measure with $\alpha\in [1,2)$ and
with or without a Brownian component.  One of the main ingredients of
the method is the sampling of the time when a Brownian motion first
exits an interval and a pre-exit location of the Brownian motion,
which may be of interest in its own right.

The \levy processes covered by the paper are among those that can be
embedded into a subordinated Brownian motion (\cite {sato:99:cup},
Chapter 30).  Similar to \cite{chi:16:spa}, the general idea is to
sample certain first exit events of the subordinated Brownian motion
and extract from them the part that belong to the \levy process under
consideration.  However, care is required to deal with unbounded
variation.  Due to the structure of the subordinated Brownian motion,
the sampling of the first exit event can be addressed by attacking two
issues in tandem.  The first one is the sampling of the first passage
event of the subordinator underlying the subordinated Brownian
motion.  For this, the results in \cite{chi:16:spa} can be directly
used.  The second issue is the sampling of the first exit time of a
Brownian motion and its value at a pre-exit time point given the
time and location of the first exit.  For this, a huge number of known
results can be used (cf.\ \cite{borodin:02:bvb, morters:10:cup}).

Section \ref{s:main} presents the main sampling procedure.  First,
the scheme of the procedure is illustrated in Section \ref
{ss:description}.  Then, in Section \ref{ss:formal}, the procedure is
formalized as Algorithm \ref{a:fet}.  In Section \ref
{ss:example}, some examples are given.  Since the sampling of the
first passage time of the subordinator involved was addressed
previously \cite {chi:16:spa}, after Section \ref{s:main},
the discussion is mostly dedicated to the sampling for the Brownian
motion.  In Section \ref {s:bm-fet}, we consider the sampling of the
first exit time of the Brownian motion by exploiting its well known
distributional properties \cite {borodin:02:bvb}.  We then consider
the sampling of the value of the  Brownian motion at a time point
before its first exit.  In Section  \ref{s:pre-exit}, we obtain some
useful results on the distribution of the pre-exit value of the
Brownian motion, which we have not been able to find in the
literature.  In Section \ref {s:sampling-pre-exit}, the sampling of
the pre-exit value is considered.  The subtlety here is the handling
of the many negative terms in the series expansion of the
density function of the pre-exit value.  Finally, some comments are
made in section \ref{s:comments}.

\section{Preliminaries} \label{s:basics}
For $c\in (0,1)$, denote by $\dgeo(c)$ the geometric distribution on
$\{0,1,2,\ldots\}$ with probability mass function (\pmf) $(1-c) 
c^k$.  For $\theta>0$, denote by $\dgamma(\theta)$ the Gamma
distribution on $(0,\infty)$ with probability density function (\pdf)
$\cf{x>0} x^{\theta-1} e^{-x} / \Gamma(\theta)$.  Denote by $\dn_c(x)$
the \pdf\ of $N(0, c)$, the normal distribution with mean 0
and variance $c$.

\subsection{\levy processes}
Let $X = (X_t)_{t\ge 0}$ be a \levy process.  Then $\mean[e^{\iunit
  \lambda X_t}] = e^{-t\che_X(\lambda)}$, $t\ge 0$, $\lambda\in
\Reals$, where for some $c, \sigma\in\Reals$ and measure $\nu$ on
$\Reals\setminus \{0\}$ which satisfies $\int \min(x^2,1)\nu(\dd
x)<\infty$,
\begin{align*}
  \che_X(\lambda)
  =
  -\iunit c \lambda + \sigma^2\lambda^2/2 + 
  \int (1-e^{\iunit \lambda x} + \iunit \lambda x\cf{|x|<1})\,
  \nu(\dd x).
\end{align*}
$\che_X$, $c$, $\sigma$, and $\nu$ are called the \Che, linear 
coefficient, Brownian coefficient, and \levy measure of $X$,
respectively.  The Radon-Nikodym derivative $\nu(\dd x)/\dd x$,
provided it exists, is called the \levy density of $X$.  If $\sigma\ne
0$ or $\int \min(|x|,1)\nu(\dd x)=\infty$, then $X$ is said to have
infinite variation, otherwise it is said to have finite variation.
In the latter case $\che_X$ can be written as
\begin{align*}
  \che_X(\lambda)
  = 
  -\iunit \delta\lambda + \int (1-e^{\iunit \lambda x})\, \nu(\dd x),
\end{align*}
with $\delta$ called the drift coefficient of $X$.  A necessary and
sufficient condition for $X$ to be nondecreasing is that it has
finite variation with $\delta\ge 0$ and $\nu((-\infty, 0))=0$.  In
this case, $X$ is called a subordinator.  Furthermore, if $Y =
(Y_t)_{t\ge 0}$ is a process independent of $X$, then $Y_X =
(Y_{X_t})_{t\ge 0}$ is called a subordinated process.  In particular,
if $Y$ is a standard Brownian motion, i.e., a \levy process with
$\che_Y(\lambda) = \lambda^2/2$, then $Y_X$ is a \levy process with
linear coefficient 0, Brownian coefficient $\sqrt\delta$, and \levy
density $\intzi \dn_s(\cdot) \nu(\dd s)$ (\cite {sato:99:cup}, Theorem
30.1).

For any open or closed set $A\subset \Reals$, according to Blumenthal
0-1 law, either 0 is regular for $A$ (with respect to $X$), i.e.,
$\varsigma := \inf\{t>0: X_t\in A\}=0$ \as, or 0 is irregular for $A$,
i.e., $\varsigma > 0$ \as\ (\cite {sato:99:cup}, p.~313).  It is easy
to see that if $X$ is not a compound Poisson process and is symmetric,
i.e., $X\sim -X$, then 0 is regular for both half-lines $(-\infty, 0)$
and $(0, \infty)$.  As a passing remark, if $X$ has infinite
variation, then 0 is regular for both half-lines (\cite
{bertoin:96:cup}, p.~167), while if $X$ has finite variation,
Bertoin's test provides a necessary and sufficient condition for 0 to
be regular for a half-line (\cite {bertoin:97:bsm}; \cite
{doney:07:sg-b}, Theorem 6.22).

\subsection{First hitting times of a Brownian motion}
Let $\bm{}=(\bm t)_{t\ge 0}$ such that $\bm t - \bm 0$ is a standard
Brownian motion.  For $a\in \Reals$, let
\begin{align*}
  \fht a = \inf\{t: \bm t = a\},
\end{align*}
while for $a>0$, let
\begin{align*}
  \fet a = \inf\{t: \bm t\in\{-a, a\} \} = \min(\fht {-a}, \fht a).
\end{align*}
Given $x\in \Reals$, denote by $\pr^x$ the probability measure under
which $\bm0\equiv x$ and by $\mean^x$ the associated expectation.  It
is well known that under $\pr^x$ the \pdf~of $\fht a$ is $\dfh{a-x}$,
where
\begin{align} \label{e:b-fht}
  \dfh a(t)
  = |a| t^{-1} \dn_t(a) = -\dn'_t(|a|) =
  \frac{|a| e^{-a^2/(2t)}}{\sqrt{2\pi} t^{3/2}}, 
\end{align}
while for $a>0$ and $x\in (-a,a)$, the \pdf~of $\fet a$, denoted by
$\dfe a(t,x)$, has series expressions
\begin{align} \label{e:b-fet1}
  \dfe a(t,x)
  &=
  \sumzi k (-1)^k [\dfh{2ka + a- x}(t) + \dfh{2ka + a + x}(t)].
  \\\label{e:b-fet2}
  &
  =
  \frac \pi{2 a^2}
  \sumzi k (-1)^k (2k+1)\exp\Cbr{
    -\frac{(2k+1)^2 \pi^2 t}{8 a^2}}\cos\frac{(2k+1)\pi x}{2a}.
\end{align}
The series in \eqref{e:b-fet1} converges rapidly for small $t>0$ but
slowly for large $t>0$, while the one in \eqref {e:b-fet2} has the
opposite property.  Eq.~\eqref {e:b-fet2} is due to the following fact
(\cite{morters:10:cup}, \S7.4).  Let $U$ be a bounded and connected
open set and $g\in C(U)$.  If $u(t,x)\in C^2((0,\infty)\times U)$ is a
bounded solution to the heat equation $\partial_t u =
(1/2) \partial_{xx} u$ with initial condition $\lim_{(t, x)\to (0,
  x_0)}  u(t,x)=g(x_0)$ for all $x_0\in U$ and Dirichlet boundary
condition $\lim_{(t,x)\to (t_0, x_0)} u(t,x)=0$ for all $t_0>0$ and
$x_0\in \partial U$, then
\begin{align*}
  u(t,x) = \mean^x [g(\bm t) \cf{t<\fet U}], \quad
  \text{where}\
  \fet U = \inf\{t: \bm t\not\in U\}.
\end{align*}
Calculating $\mean^x[g(\bm t)\cf{t < \fet U}]$ then boils down to
solving the heat equation with the specified initial and boundary
conditions.  In particular,  if $U = (-a,a)$, by separating the
variables $t$ and $x$ and considering the eigenfunctions of
$(1/2) \partial_{xx}$ and those of $\partial_t$ with the same
eigenvalues,
\begin{align} \label{e:heat-eigen} 
  u(t,x)
  =
  \sumzi k \alpha_k
  \exp\Cbr{-\frac{k^2\pi^2 t}{8 a^2}} \varphi_k\Grp{\frac x{2a}},
\end{align}
where $\varphi_k(x)$ is $\cos (k\pi x)$ for odd $k$ and $\sin (k \pi
x)$ for even $k$, and $\alpha_k = a^{-1}\int^a_{-a} g(x)\varphi_k(x)
\, \dd x$.  To get Eq.~\eqref {e:b-fet2}, apply \eqref {e:heat-eigen}
to $g(x)\equiv 1$ to yield $\pr^x\{\fet a > t\}$, and then
differentiate the result in $t$.

\section{Main results} \label{s:main}
Let $X$ be a symmetric \levy process.  Suppose the Brownian
coefficient of $X$ is $\delta\ge 0$ and its \levy density is
\begin{align*}
  \lambda(x) = \lambda_0(x)\cf{|x|<r}, \quad 0<r\le\infty,
\end{align*}
such that $\lambda_0(x)$ is the \levy density of a subordinated
process $Z = \bm S$, where $\bm{}$ is a standard Brownian motion and
$S$ a subordinator with drift coefficient $\delta^2$ independent of
$\bm{}$.  Then $Z$ is symmetric and its Brownian coefficient is
$\delta$ as well.

Given interval $I = (b,c)$ with $-\infty\le b<0<c\le \infty$, the
first exit time of $X$ out of $I$ is 
\begin{align*}
  T_I = \inf\{t>0: X_t \not\in I\}.
\end{align*}
The value and jump of $X$ at the time of exit are also important
information.  We shall consider the sampling of the triplet  $(T_I,
X_{T_I-}, X_{T_I})$, where for $t>0$, $X_{t-}$ is the left limit of
$X$ at $t$.

By right-continuity of $X_t$, $T_I>0$.  As long as $X_t\not\equiv
0$ and $\min(|b|, c)<\infty$, $T_I$ is finite.  This is because
either i) $\lim X_t = \infty$ \as, or ii) $\lim X_t = -\infty$ \as, or
iii) $\limsup X_t = - \liminf X_t = \infty$ \as\ (\cite
{bertoin:96:cup}, Theorem VI.16). Since $X$ is symmetric, iii)
must hold, so $T_I <\infty$.  As a passing remark, for any
\levy process, Erickson's test provides a necessary and sufficient
condition on which of the three cases holds (\cite
{doney:07:sg-b}, Theorem 4.15 and p.~64).

In this section, it is always assumed that
\begin{align} \label{e:X-non-CP}
  \text{$X$ is not a compound Poisson process},
\end{align}
which is equivalent to $S$ not being a compound Poisson subordinator.
Since $X$ is symmetric, then with respect to $X$, $c$ is regular for
$(c, \infty)$ and $b$ is regular for $(-\infty, b)$, so
\begin{align} \label{e:fet-X}
  T_I = \inf\{t>0: X_t\not\in [b,c]\}.
\end{align}
In particular, if $0<c<\infty$, $T_{(-\infty,c)}$ is the first passage
time of $X$ across $c$, i.e., 
\begin{align*}
  T_{(-\infty, c)} = \inf\{t>0: X_t > c\}.
\end{align*}

By definition, $Z_{t-} = \lim_{u\to t-} \bm{S_u}$.  Under assumption
\eqref{e:X-non-CP},
\begin{align} \label{e:LC}
  Z_{t-} = \bm{S_{t-}}.
\end{align}
Indeed, as $u\to t-$, since $S$ is not a compound Poisson process and
hence strictly increasing, $S_u\to s := S_{t-}$, given $\bm{S_u}\to
\bm{s-} = \bm s$ by continuity of $\bm{}$.  Then \eqref{e:LC}
follows.

\subsection{Description} \label{ss:description}
\begin{figure}[t]
  \caption{$Z = \bm S$.  (a) $r/2\le Z_{\fht{}} \le Z_{\fht{}-}+r$, so
    $X_{\fht{}} = Z_{\fht{}}$ and $\fht{}$ is the first exit time of
    $X$ and $Z$ out of $(-r/2, r/2)$; (b) $-r/2\le Z_{\fht{}}\le r/2$,
    so $X_{\fht{}} = Z_{\fht{}}$, but $\fht{}$ is not an exit time of
    $X$ out of $(-r/2, r/2)$; (c) $Z_{\fht{}} > Z_{\fht{}-}+r$, so
    $X_{\fht{}} = X_{\fht{}-}$, $\fht{}$ is not the first exit time of
    $X$ out of $(-r/2, r/2)$ but is that of $Z$ (from the top); (d)
    $Z_{\fht{}}<-r/2$}
  \label{f:fet}
  \begin{center}
    \setlength{\unitlength}{1.2mm}
    \begin{picture}(105,58)(-1,10)
      \put(0,5){\scalebox{.84}[.72]{\includegraphics{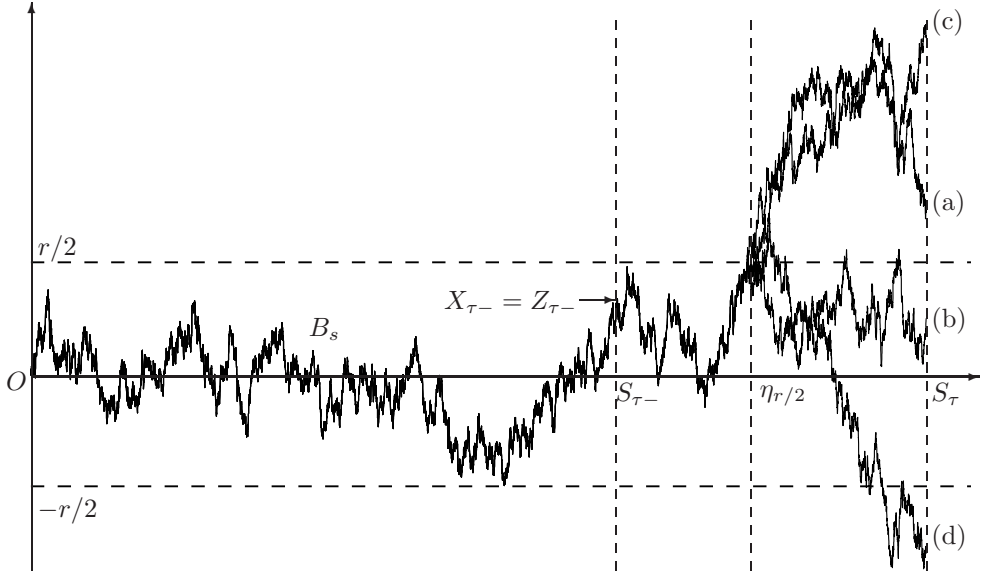}}}
      \put(0.3,5){\vector(0,1){63}}
      \put(0.3,26.5){\vector(1,0){105}}

      \put(1,11){\small $-r/2$}
      \put(1,40){\small $r/2$}
      \put(-2.5,25){\small $O$}

      \put(31, 31){\small $\bm s$} 
      \put(46, 34){\small $X_{\fht{}-} = Z_{\fht{}-}$}
      \put(61, 35){\vector(1,0){4}}

      \multiput(65, 5)(0,2){31}{\line(0,1){1}}
      \multiput(80, 5)(0,2){31}{\line(0,1){1}}
      \multiput(99.5, 5)(0,2){31}{\line(0,1){1}}

      \put(65,23.8){\small $S_{\fht{}-}$}
      \put(81,24.3){\small $\fet {r/2}$}
      \put(100,23.8){\small $S_{\fht{}}$}

      \put(100, 45){\small (a)}
      \put(100, 32){\small (b)}
      \put(100, 65){\small (c) }
      \put(100, 8){\small (d)}
    \end{picture}
  \end{center}
\end{figure}

To sample $(T_I, X_{T_I-}, X_{T_I})$, the approach is to embed
$X$ into $Z = \bm S$ and exploit a sequence of hitting or passage
events of $\bm{}$ and $S$.  By embedding it means that, by identifying
$Z$ with $X + V$, $X$ is a part of $Z$, where $V$ is a
compound Poisson process with \levy density $\lambda_0(x)\cf{|x|\ge
  r}$ independent of $X$.  Equivalently,
\begin{align} \label{e:X-Z}
  X_t = Z_t - \sum_{s\le t} \jump Z_s \cf{|\jump Z_s|\ge r},
\end{align}
where $\jump Z_s = Z_s - Z_{s-}$ is the jump of $Z$ at $s$.  The
issue is how to identify the jumps of $V$ so that random
variables purely due to $X$ can be extracted from those of $Z$.

Figure \ref{f:fet} illustrates the idea.  Suppose $(-r/2, r/2) \subset
I$.  At time $\fet{r/2}$, $\bm{}$ hits the boundary of $(-r/2,r/2)$,
say at $r/2$.  Then
\begin{align} \label{e:pre-exit}
  |\bm{s_2} - \bm{s_1}| < r \quad\text{for all}\ 0< s_1 < s_2
  \le\fet{r/2}.
\end{align}
Let
\begin{align*}
  \fht{} = \inf\{t>0: S_t>\fet{r/2}\}.
\end{align*}
From \eqref{e:LC} and \eqref{e:pre-exit}, it follows that for $0 < t <
\fht{}$, $|\jump Z_t| <r$.  Thus, in $[0,\fht{})$, $Z$ only has jumps
of size strictly less than $r$.  Then $X_t = Z_t$ on $[0, \fht{})$ and
the first jump in $V$ can only appear at $\fht{}$.  Figure \ref
{f:fet} shows a scenario where $S$ has a jump at $\fet{r/2}$.  Because
the potential measure of $S$ is diffuse, $S_{\fht{}-} < \fet{r/2} <
S_{\fht{}}$ (\cite {bertoin:96:cup}, Propositions I.15 and III.2).
Then $X_{\fht{}-} = Z_{\fht{}-} = \bm{S_{\fht{}-}}$ and by \eqref
{e:LC},  $\jump Z_{\fht{}} = \bm{S_{\fht{}}} - \bm{S_{\fht{}-}}$.  If
$|\jump Z_{\fht{}}|<r$, then the jump belongs to $X$, giving
$X_{\fht{}} = Z_{\fht{}}$.  This is the case in (a) and (b) in Figure
\ref{f:fet}.  In (a), since $X_{\fht{}} \ge r/2$, $\fht{}$ is the
first time of $X$ out of $(-r/2,r/2)$, whereas in (b), it is not.  If
$|\jump Z_{\fht{}}|\ge r$, then the jump belongs to $V$, giving
$X_{\fht{}} = X_{\fht{}-}$.  This is the case in (c) in Figure \ref
{f:fet}, where $\fht{}$ is the first exit time of $Z$ out of $(-r/2,
r/2)$ but not that of $X$.  Also, as (d) shows, even if $\bm{}$ first
exits $(-r/2, r/2)$ at the top boundary point, $Z$ and $X$ may first
exit at the lower boundary. 

From the above description, it is seen that the following quantities
have to be sampled in sequel,
\begin{enumerate}
\item[a)] $\fet{r/2}$, the first hitting time of $\bm{}$ at $\pm r/2$;

\item[b)] the triplet $(\fht{}, S_{\fht{}-}, S_{\fht{}})$, where
  $\fht{}$ the first passage time of $S$ across $\fet{r/2}$; and

\item[c)] $\bm{S_{\fht{}-}}$ and $\bm{S_{\fht{}}}$.
\end{enumerate}
Because $S$ and $\bm{}$ are independent, b) boils down to the sampling
of the first passage of $S$ across any fixed level.  This is addressed
in \cite{chi:16:spa} for several important classes of subordinators.
Items a) and c) will be considered in following sections.  The
sampling of $\bm{S_{\fht{}-}}$ boils down to that of $\bm T$
conditional on $(\fet{r/2}, \bm{\fet{r/2}}) = (T+t, \pm r/2)$ for
fixed $T,t>0$.  On the other hand, by the strong Markov property of
Brownian motion, $\bm{S_{\fht{}}}$ can be simply sampled from a normal
distribution.

Figure \ref{f:fet} just illustrates a single iteration of the sampling
procedure.  If at the end of the iteration, $X$ has yet exited $I$,
then the procedure is renewed at time $\fht{}$.  The iteration
continues until an exit occurs.  Note that $\fht{}$ is not a stopping
time of $X$ or $Z$ as it depends on $\fet{r/2}$, information not
available via $X$ or $Z$.  However, conditional on $\bm{}$, $\fht{}$
is a stopping time of $S$, justifying iterating by renewal.  Indeed,
the procedure can be thought of as one with the entire path of $\bm{}$
being sampled in advance and $S$ being the only random process during
the run.  In this setting, the sampling of $\fet{r/2}$, $Z_{\fht{}-}$
and $Z_{\fht{}}$ can be regarded as a subroutine to retrieve data from
the path of $\bm{}$. 

There are two simpler cases not covered so far.  First, if $S$ has a
positive drift, then it may creep across $\fet{r/2}$, i.e.,
$S_{\fht{}-} = S_{\fht{}} = \fet{r/2}$ (\cite{bertoin:96:cup}, Theorem 
III.5).  Clearly in this case $X_{\fht{}-} = X_{\fht{}}$ is equal to the
point where $\bm{}$ exits $(-r/2,r/2)$.  Second, suppose we wish to
sample $(\varsigma, X_{\varsigma-}, X_\varsigma)$ instead, where
$\varsigma = \min(T_I, T_0)$ with $T_0<\infty$ a fixed terminal point.
This allows, for example, the sampling of $X_{T_0}$ when $I = \Reals$.
If $T_0<\fht{}$, then $S_{T_0-} = S_{T_0}$ instead of $S_{\fht{}-}$
and $S_{\fht{}}$ should be sampled conditional on $S_{T_0} <
\fet{r/2}$ (cf.\ \cite {chi:16:spa}) and then $X_{T_0} = Z_{T_0} =
\bm{S_{T_0}}$ is sampled conditional on $(\fet{r/2},
\bm{\fet{r/2}})$.

\subsection{Formal procedure} \label{ss:formal}
The description in Section \ref{ss:description} is formalized as
Algorithm \ref {a:fet} with additional detail taken into account.  In
the procedure, $I = (b,c)$ and $\varsigma$ is defined at the end of
Section \ref{ss:description}.  It is quite routine to extend the
procedure to sample the first exit event of $X+Y$, where $Y$ is a
compound Poisson process independent of $X$ (cf.\ \cite{chi:16:spa}).
For brevity, the detail of the extension is omitted.

\begin{algorithm}[h]
  \caption{Sampling $(\varsigma, X_{\varsigma-}, X_\varsigma)$}
  \label{a:fet}
  \begin{algorithmic}[1]
    \REQUIRE $-b, c, r, T_0$ all in $(0, \infty]$ such that if
    $|b|=c=\infty$ then $r<\infty$ and $T_0<\infty$

    \STATE $T\ot 0$, $W\ot 0$

    \REPEAT

    \STATE $a\ot \min(r/2, W-b, c-W)$  \label{a:fet-a}

    \STATE Sample $h$ from the distribution of $\fet a$  and sample
    $y$ from $\{-a,a\}$ uniformly \label{a:fet-y}

    \STATE Sample $(t, s_-, s_+)$ from the distribution of $(\theta,
    S_{\theta-}, S_\theta)$, where $\theta = \min(\fht{}, T_0 - T)$
    with $\fht{}=\inf\{t>0:  S_t> h\}$ \label{a:fet-fpt}

    \STATE Sample $x$ from the conditional distribution of $\bm{s_-}$
    given $(\fet a, \bm{\fet a}) = (h,y)$ \label{a:fet-x}

    \IF {$s_-<s_+$}

    \STATE Sample $u$ from $N(0, s_+- s)$  \label{a:fet-zeta}

    \STATE $d\ot y- x + u$, $D \ot d\cf{|d|\le r}$ \label{a:fet-D}

    \ELSE

    \STATE $D\ot 0$ \label{a:fet-D0}

    \ENDIF 

    \STATE $T\ot T + t$, $W\ot W + x + D$  \label{a:fet-update}

    \UNTIL $T=T_0$ or $W\not\in (b,c)$

    \RETURN $(T, W-D, W)$
  \end{algorithmic}
\end{algorithm}

In Algorithm \ref{a:fet}, $T$ is a value such that $X_t\in (b,c)$ for
all $t\in (0,T)$ and $W= X_T$.  In each iteration $(T, W)$ is updated
following the description in Section \ref{ss:description}.  Some
explanations are in order.  The value of $a$ on line \ref {a:fet-a}
makes sure the interval $(W-a, W+a)$ is in $(b,c)$.  Since $\bm{}$ is
a standard Brownian motion, $y$ on line \ref {a:fet-y} follows the
distribution of $\bm{\fet a}$ conditional on $\fet a$.  If $s_-<s_+$,
then as explained in Section \ref {ss:description}, $t < T_0 - T$ and
$s_-<h<s_+$.  By the strong Markov property of $\bm{}$ and the
independence between $\bm{}$ and $S$, conditional on $S_\theta$,
$\bm{S_\theta} - \bm{\fet a}$ is independent of $(\bm t)_{t\le\fet a}$
and follows $N(0, \bm{S_\theta} - \bm{\fet a})$.  Thus $u$ on line
\ref {a:fet-zeta} follows the distribution of $\bm{S_\theta} -
\bm{\fet a}$, and $d$ and $D$ on line \ref{a:fet-D} follow the joint
distribution of $\jump Z_{\fht{}}$ and $\jump X_{\fht{}}$.  On the
other hand, if $s_-=s_+$, then either $s_-=h$ or $T_0 - T < t$,
resulting in no jump and line \ref{a:fet-D0}.  Line \ref
{a:fet-update} updates the values of $T$ and $W$.  It is clear that
the iteration stops only when $T=T_0$ or $X_T\in (b,c)$, i.e.\
$T=\varsigma$.  When the iteration stops, since $D = \jump X_T$,
$(T,W-D, W)$ follows the distribution of $(\varsigma, X_{\varsigma-},
X_\varsigma)$.

\begin{theorem} \label{t:fet}
  The iteration in Algorithm \ref{a:fet} eventually stops \as\ and its
  output follows the distribution of $(\varsigma, X_{\varsigma-},
  X_\varsigma)$.
\end{theorem}
\begin{proof}
  It suffices to show that the iteration eventually stops \as.  Let
  $t_0=0$, $w_0=0$, and for $n\ge 1$, $t_n$ and $w_n$ the values of
  $T$ and $W$ at the end of the $n\th$ iteration, respectively.   Let
  $s_n = S_{t_n}$ and $z_n = Z_{t_n}$.  It is easy to see that both
  $t_n$ and $s_n$ are strictly increasing and $t_n \le \varsigma <
  \infty$.  Define event
  \begin{align*}
    E = \{\text{the iteration in Algorithm \ref{a:fet} does not
      stop}\}.
  \end{align*}
  In $E$, there are infinitely many $t_n$.   Let $t_* = \lim t_n$
  and $s_* = \lim s_n$.  Since conditional on $\bm{}$, $t_n$ are
  stopping times of $S$, $s_* = S_{t_*}$ \as\ (\cite {bertoin:96:cup},
  Proposition I.7).  Then by continuity of $\bm{}$, 
  \begin{align*}
    \sup_{t_n \le t<t'\le t_*} |Z_{t'} - Z_t|
    \le
    2\sup_{t_n \le t \le t_*} |Z_t - \bm {s_*}|
    \le
    2\sup_{s_n \le s\le s_*} |\bm s - \bm {s_*}| \to 0, \quad n\toi.
  \end{align*}
  As a result, $z_n\to \bm{s_*}$ and there is $N\ge 0$ such that
  $|\jump Z_t|<r$ on $[t_N, t_*]$.  By renewal argument, we can assume
  $N=0$ without loss of generality.  Then $X_t = Z_t$ on $[0, t_*]$,
  in particular,
  \begin{align} \label{e:wzb}
    w_n = z_n \to \bm{s_*}.
  \end{align}
  On the other hand, let $a_n$, $h_n$, and $\fht n$ be the values of
  $a$, $h$, and $\fht{}$ in the $n\th$ iteration, respectively.  Then
  $t_n = t_{n-1} + \min(\fht n, T_0 - t_{n-1})$.  In $E$, $t_n<T_0$.
  Then $\fht n = t_n - t_{n-1}\to 0$.  Since $\fht n \sim \inf\{t>0:
  S_t>h_n\}$, then $h_n\to 0$.  Since $h_n\sim \inf\{t>0: \bm t>
  a_n\}$, then $a_n\to 0$.  Consequently, $\min(w_n - b,   c-w_n)\to
  0$.  Combined with \eqref{e:wzb}, this yields $w_n = z_n$ either
  converges to $b$ or to $c$.  Without loss of generality, suppose the
  limit is $c$.  By using renewal argument again, we can assume that
  $c-w_n < w_n - b$ for all $n\ge 0$.  Note that by renewing
  $(X_t,Z_t)$ at $t=t_{n-1}$, $\bm s$ is renewed at $s=s_{n-1}$.  It
  follows that $s_{n-1} + h_n$ is the first $s>s_{n-1}$ such that $\bm
  s$ hits the boundary of $(w_n - a_n, c)$ and, with $S$ being
  non-compound Poisson, $t_n$ is the first $t\ge t_{n-1}$ such that
  $S_t\ge s_{n-1} + h_n = S_{t_{n-1}} + h_n$.  Now $\bm {s_{n-1} +
    h_n}$ is either $w_n - a_n$ or $c$, each with probability $1/2$.
  If $\bm{s_{n-1} + h_n}=c$, then $\bm{s_n} - \bm{s_{n-1} + h_n} = w_n
  - c<0$.  By strong Markov property of $\bm{}$, $\bm{s_n} -
  \bm{s_{n-1} + h_n}$ are independent normal random variables with
  mean 0, possibly degenerate, so the probability that the procedure
  does \emph{not\/} stop at the $n\th$ iteration is at most
  $\pr\{\bm{s_{n-1} + h_n} = w_n - a_n\} + \pr\{\bm{s_{n-1}+h_n} = c,
  z_n - c<0\} = 3/4$.  It follows that the probability to have
  infinite iterations is 0.  Then $\pr(E) = 0$.
\end{proof}

\subsection{Examples} \label{ss:example}
\begin{example} \rm
  Let $X$ be a symmetric \levy process with \levy density $\lambda(x)
  = c\cf{0<|x|< r} x^{-\alpha-1}$, $r\in (0,\infty)$.  If $\alpha\in
  (0,1)$ and the Brownian coefficient $\delta$ of $X$ is 0, then $X$
  has finite variation.  In this case, it has been shown that
  $(\varsigma, X_{\varsigma-}, X_\varsigma)$ can be sampled exactly
  \cite{chi:12:advap, chi:16:spa}.  On the other hand, if $\delta
  \ne0$ or $\alpha\in [1,2)$, $X$ has infinite variation.  In this
  case, $X$ can be embedded into $Z=\bm S$ as in \eqref{e:X-Z}, where
  $S$ is a subordinator with drift coefficient $\delta^2$ and \levy
  density $c\cf{x>0} x^{-\alpha/2 - 1} / d_{\alpha/2}$, where for
  $z>0$, $d_z = \Gamma(z + 1/2) 2^z / \sqrt\pi$.  The first passage
  event of $S$ can be sampled exactly \cite{chi:12:advap, chi:16:spa}.
  By combining this with the results in following sections, Algorithm
  \ref{a:fet} can sample $(\varsigma, X_{\varsigma-}, X_\varsigma)$.
\end{example}

\begin{example}\rm
  Suppose instead that $S$ is a subordinator with exponentially tilted
  \levy density $c\cf{x>0} e^{-s x} x^{-\alpha/2-1} / d_{\alpha/2}$,
  where $s>0$ and $\alpha\ge 1$.  The first passage event of $S$ can
  be sampled exactly \cite{chi:12:advap, chi:16:spa}.  On the other
  hand, the \levy density of $\bm S$ is
  \begin{align*}
    \lambda_0(x)
    &=
    \frac{c}{d_{\alpha/2}} \intzi 
    \frac{e^{-x^2/2 u}}{\sqrt{2\pi u}} 
    e^{-s u} u^{-\alpha/2-1}\,\dd u
    \\
    &=
    c \cf{x>0} x^{-\alpha - 1}
    +
    \frac{c}{d_{\alpha/2}} \int^1_0
    \frac{e^{-x^2/2 u}}{\sqrt{2\pi u}} 
    (e^{-s u} - 1) u^{-\alpha/2-1}\,\dd u + O(1)
    \\
    &=
    c \cf{x>0} x^{-\alpha - 1}
    -
    \frac{c s}{d_{\alpha/2}} \int^1_0
    \frac{e^{-x^2/2 u}}{\sqrt{2\pi}} u^{-\alpha/2-1/2}\,\dd u +O(1+s^2),
  \end{align*}
  where the implicit constant in $O(1+s^2)$ only depends on $(\alpha,
  c)$.  By variable substitute $v = x^2/(2u)$,
  \begin{align*}
    \int^1_0
    \frac{e^{-x^2/2 u}}{\sqrt{2\pi}} u^{-\alpha/2-1/2}\,\dd u 
    &=
    \frac 1{\sqrt{2\pi}} \Grp{\frac 2{x^2}}^{\alpha/2 - 1/2} 
    \int^\infty_{x^2/2} e^{-v} v^{\alpha/2 - 3/2}\,\dd v
    \\
    &\sim
    \begin{cases}
      d_{\alpha/2-1} x^{-\alpha+1} & \alpha\in(1,2)
      \\
      -\ln|x|/\sqrt{2\pi} & \alpha = 1
    \end{cases} \quad \text{as}\ x\to 0.
  \end{align*}
  Therefore, if $\alpha\in (0,1)$ and $X$ has \levy density
  $\lambda(x) = c\cf{0<|x|<r} (1 - C x^2) x^{-\alpha-1}$, where $C>0$,
  then by choosing $s>0$ large enough and $r'\in (0,r]$ small enough,
  $\lambda$ can be written as $\lambda(x) = \cf{0<|x|<r'} \lambda_0(x)
  + \chi(x)$, where $\chi$ is the \levy density of a compound Poisson
  process.  Then, as noted just before Algorithm \ref{a:fet}, a
  routine extension of the procedure is able to sample $(\varsigma,
  X_{\varsigma-}, X_\varsigma)$.  The case $\alpha=1$ can be similarly
  dealt with.
\end{example}

\begin{example}\rm
  Let $X$ be a symmetric \levy process with $\lambda(x) = c \cf{0<|x|
    < r} e^{-x/\beta} / x$ and Brownian coefficient $\delta$.  It has
  been shown that if $\delta = 0$, then $(\varsigma, X_{\varsigma-},
  X_\varsigma)$ can be sampled exactly \cite {chi:12:advap,
    chi:16:spa}.  Note that $\lambda$ is a truncated version of the
  \levy density of $U-D$, where $U$ and $D$ are independent Gamma
  processes with \levy density $\lambda_0(x) = c\cf{x>0}
  e^{-x/\beta}/x$.  It is well known that $U-D \sim \bm V$, where $V$
  is a Gamma process with \levy density $c\cf{x>0} e^{-x/2\beta}/x$
  (cf.\ \cite{glasserman:04:sv-ny}, p.~143-144).  Assume now that
  $\delta>0$.  Then $X$ can be embedded in $Z = \bm S$, where $S_t =
  \delta^2 t + V_t$.  The exactly sampling of the first passage event
  of $S$ has been shown in \cite{chi:16:spa}.  Then Algorithm
  \ref{a:fet} can be used to sample $(\varsigma, X_{\varsigma-},
  X_\varsigma)$.
\end{example}

\section{Sampling of first exit time of a Brownian motion}
\label{s:bm-fet}
In this section, denote
\begin{align*}
  \psi(x) = x e^{-x^2/2}, \quad C_0 = 2/\sqrt e.
\end{align*}
Then $x e^{-x y/2} \le C_0 e^{-y/2}$ for all $x$, $y\ge 1$, in
particular,
\begin{align} \label{e:h}
  \psi(x y) = x y e^{-x^2 y^2/2} \le C_0 y e^{-x y^2/2}.
\end{align}

For $s>0$ and $k\ge 0$, denote
\begin{align} \label{e:dk}
  d_k(s) = \psi((4k+1)\sqrt{2s}) - \psi((4k+3)\sqrt{2s}).
\end{align}
Given $a>0$, for $t>0$, by \eqref{e:b-fet2},
\begin{align} \label{e:b-fet2a}
  \dfe a(t,0)
  &=
  \frac \pi{2a^2\sqrt {2 x}}
  \sumzi k d_k(x), \qquad
  \text{with}\ 
  x = \frac{\pi^2 t}{8 a^2},
\end{align}
and by \eqref{e:b-fet1},
\begin{align} \label{e:b-fet1a}
  \dfe a(t,0)
  &=
  \sqrt{\frac 2\pi} \frac {2 y}{a^2}
  \sumzi k d_k(y),
  \qquad
  \text{with}\ 
  y = \frac{a^2}{2t}.
\end{align}

Let $X = \pi^2\fet a/(8 a^2)$.  Then by \eqref{e:b-fet2a}, the \pdf~of
$X$ is
\begin{align} 
  \dfh X(x)
  &=
  (8 a^2/\pi^2) \times \dfe a(8 a^2 x/\pi^2)
  \nonumber \\
  &=
  A \times
  \frac{e^{-x}}{\pr\{\xi\ge \pi^2/8\}}
  \times \frac{1-e^{-\pi^2/2}}{1 - e^{-4x}}
  \sumzi k
  (1-e^{-4x}) e^{-4 k x} \times
  \frac{d_k(x)}{C_0 \sqrt{2 x} e^{-(4 k + 1)x}},
  \label{e:fet-b3}
\end{align}
where $\xi \sim \dgamma(1)$ and
\begin{align} \label{e:b-fet-A}
  A = 
  \frac{4C_0 \pr\{\xi\ge \pi^2/8\}}{\pi (1-e^{-\pi^2/2})}.
\end{align}
For $t\ge a^2$, $x\ge \pi^2/8$, so $e^{-x}/e^{-\pi^2/8}$ is the
\pdf~of $\xi$ at $x$ conditional on $\xi\ge \pi^2/8$ and
$(1-e^{-\pi^2/2})/(1-e^{-4x})\le 1$.  Next, $(1-e^{-4x}) e^{-4kx} =
\pr\{\kappa = k\}$ for $\kappa \sim \dgeo(e^{-4 x})$.  Finally, since
$\sqrt{2 x}>1$ and $\psi$ is positive and strictly decreasing on
$[1,\infty)$,
\begin{align*}
  0< \frac{d_k(x)}{C_0 \sqrt{2 x} e^{-(4 k + 1) x}} <1,
\end{align*}
where the second inequality uses \eqref{e:h}.  As a result, \eqref
{e:fet-b3} implies that rejection sampling can be used to sample
$X$ conditional on $X\ge \pi^2/8$, and hence to sample
$\fet a$ conditional on $\fet a\ge a^2$.

Likewise, by \eqref{e:b-fet1a}, the \pdf~of $Y = a^2/(2\fet a)$ is
\begin{align} 
  \dfh Y(y)
  &=
  a^2/(2 y^2) \times \dfe a(a^2/(2 y),0)
  \nonumber\\
  &=
  B
  \times
  \frac{e^{-y}/\sqrt{\pi y}}{\pr\{\zeta>1/2\}}
  \times
  \frac{1-e^{-2}}{1-e^{-4y}}
  \sumzi k (1-e^{-4y}) e^{-4 k y} 
  \times
  \frac{d_k(y)}{C_0 \sqrt{2 y} e^{-(4 k + 1)y}},
  \label{e:fet-a3}
\end{align}
where $\zeta\sim \dgamma(1/2)$ and
\begin{align} \label{e:b-fet-B}
  B = \frac{2 C_0 \pr\{\zeta>1/2\}}{1-e^{-2}}.
\end{align}
As a result, rejection sampling can be used to sample $Y$ conditional
on $Y \ge 1/2$, and hence to sample $\fet a$ conditional on $\fet a
\le a^2$.

Recall that $\pr^x$ denotes probability measure under which $\bm0\equiv
x$.  The above results lead to the rejection sampling Algorithm
\ref{a:bfet}.

\begin{algorithm}
  \caption{Sampling $\fet a$ under $\pr^0$}
  \label{a:bfet}
  \begin{algorithmic}[1]
    \REQUIRE $a\in (0,\infty)$, $C_0 = 2/\sqrt e$, $d_k(\cdot)$ as in
    \eqref{e:dk}, $A$ and $B$ as in \eqref{e:b-fet-A} and \eqref
    {e:b-fet-B}, respectively
    \WHILE{(1)}
    \STATE Sample $\eno U 5$ \iid~$\sim \dunif(0,1)$
    \IF {$U_1\le A/(A+B)$}
    \STATE Sample $\xi\sim \dgamma(1)$ conditional on $\xi\ge \pi^2/8$,
    then sample $\kappa\sim \dgeo(e^{-4\xi})$
    \IF {$U_2 (1-e^{-4\xi}) \le 1-e^{-\pi^2/2}$ and $C_0 U_3 \sqrt{2
        \xi} e^{-(4\kappa + 1)\xi} \le d_\kappa(\xi)$
    }
    \RETURN $8 a^2\xi / \pi^2$
    \ENDIF
    \ELSE
    \STATE Sample $\zeta\sim \dgamma(1/2)$ conditional on $\zeta>1/2$,
    then sample $\kappa \sim \dgeo(e^{-4\zeta})$
    \IF {$U_4 (1 - e^{-4\zeta}) \le 1-e^{-2}$ and
      $C_0 U_5 \sqrt{2 \zeta} e^{-(4\kappa + 1) \zeta} \le
      d_\kappa(\zeta)$}
    \RETURN $a^2/(2\zeta)$
    \ENDIF
    \ENDIF
    \ENDWHILE
  \end{algorithmic}
\end{algorithm}

\section{Distribution of pre-exit location of a Brownian motion}
\label{s:pre-exit}
We need a few more properties of the first exit event of a Brownian
motion.  For $t>0$ and $x\in (-a,a)$, it is known that
\begin{align*}
  \pr^x\{\fet a \in \dd t,\, \bm{\fet a} = \pm a\}&
  = \dfe a^\pm(t,x)\,\dd t,
\end{align*}
where
\begin{align}\label{e:b-fet-top1}
  \dfe a^+(t,x)
  =
  \dfe a^-(t,-x)
  =
  \sumzi k [\dfh {4ka+a-x}(t) - \dfh {4ka+3a+x}(t)]
\end{align}
(\cite{borodin:02:bvb}, p.~212, 3.0.6).  We also have the following.

\begin{prop} \label{p:b-fet}
  For all $t>0$ and $x\in (-a,a)$
  \begin{align}
    \label{e:b-fet-top2}
    \dfe a^+(t,x)
    =
    \dfe a^-(t,-x)
    &=
    \frac 1 2 \dfe a(t,x)
    -\frac\pi{2a^2}
    \sumoi k(-1)^k k \exp\Cbr{-\frac{k^2\pi^2 t}{2 a^2} } \sin
    \frac{k\pi x}{a}.
  \end{align}
  Furthermore, $\dfe a^\pm(t,x)>0$.
\end{prop}

\begin{proof}
  To start with, by the Markov property of $\bm{}$,
  \begin{align*}
    \pr^x\{\fet a > t,\, \bm{\fet a} = a\}
    =
    \mean^x [\cf{t<\fet a} \cf{\bm{\fet a}=a}]
    =
    \mean^x[\cf{t<\fet a} \pr^{\bm t}\{\bm{\fet a}'=a\}],
  \end{align*}
  where $\bm{}'$ is an \iid~copy of $\bm{}$.  By $\pr^y(\bm{\fet a}=a)
  = (a+y)/(2a)$, $y\in (-a,a)$ (\cite {borodin:02:bvb}, p.~212,
  3.0.4),
  \begin{align*}
    \pr^x\{\fet a> t,\, \bm{\fet a}=a\}
    =
    \mean^x[\cf{t<\fet a} (a+\bm t)/(2a)]
    =
    \frac 1 2 \pr^x\{t<\fet a\} +
    \frac1{2a}\mean^x[\bm t\cf{t<\fet a}].
  \end{align*}
  Put $u(t,x) = \mean^x[\bm t\cf{t<\fet a}]$.  Using \eqref
  {e:heat-eigen} and the fact that $x$ is antisymmetric on $(-a,a)$,
  \begin{align*}
   u(t,x)
   &=
   \sumoi k \alpha_k \exp\Cbr{-\frac{k^2\pi^2 t}{2a^2}} \sin\frac{k\pi
     x}{a},
  \end{align*}
  where
  \begin{align*}
    \alpha_k
    = \frac 1 a \int^a_{-a} x \sin\frac{k\pi x}{a}\,\dd x
    = (-1)^{k-1} \frac {2a}{k\pi}.
  \end{align*}
  Then 
  \begin{align*}
    \pr^x\{\fet a>t,\, \bm{\fet a}=a\}
    = \frac 1 2\pr^x\{\fet a>t\} - \frac 1\pi
    \sumoi k \frac{(-1)^k} k 
    \exp\Cbr{-\frac{k^2\pi^2 t}{2a^2}} \sin\frac{k\pi x}{a}.
  \end{align*}
  Differentiating both sides in $t$ then yields \eqref{e:b-fet-top2}.

  To show $\dfe a^+(t,x)>0$ for all $x\in (-a,a)$, regard $x$ as a
  parameter while $t$ the only variate.  Then denote $\dfe {x,
    a}^\pm(t) = \dfe a^\pm(t,x)$.  By \eqref{e:b-fet-top1}, $\dfe{x,
    a}^+ = \dfh {a-x} - \dfh {2a} * \dfe{x,a}^-$, $\dfe{x,a}^- = \dfh
  {a+x} - \dfh {2a}* \dfe{x,a}^+$.  Then $\dfe{x,a}^- \le \dfh {a+x}$,
  so $\dfe{x,a}^+ \ge \dfh {a-x} - \dfh {2a}* \dfh {a+x} = \dfh {a-x}
  - \dfh {3 a + x}$.  From \eqref{e:b-fht}, it follows that $\dfe{x,
    a}^+(t)>0$ for all small $t>0$.  Likewise, $\dfe{x,a}^-(t)$ for
  all small $t>0$.  Assume $\{t>0: \dfe{x,a}^+(t) = 0\} \ne \emptyset$
  and let $t_0$ be the infimum of the set.  Then $t_0>0$.  Since
  $\dfe{x,a}^+$ is smooth, $\dfe{x,a}^+(t_0)=0$.  Fix $c\in (x, a)$.
  Under $\pr^x$, in order for $\bm{}$ to reach $a$ before reaching
  $-a$, it may first reach $c$ before reaching $-a$, then starting at
  $c$, return to $x$ before reaching $a$, and finally, starting at $x$
  again, reach $a$ before reaching $-a$.  As a result, $\dfe{x,a}^+
  \ge \dfe{x',a'}^+ * \dfe{x'',a''}^- * \dfe{x,a}^+$, where $x' =
  x+(a-c)/2$, $a' = (c+a)/2$, $x'' = c - (a+x)/2$, $a'' = (a-x)/2$.
  Since $\dfe{x,a}^+(t) > 0$ for all $0<t<t_0$, it follows that
  $\dfe{x',a'}^+ * \dfe{x'',a''}^-(t)=0$ for all $t<t_0$.  However
  $\dfe{x',a'}^+(t)>0$ and $\dfe{x'', a''}^-(t)>0$ for all small $t$.
  The contradiction shows $\dfe{x,a}^+$ is strictly positive.
  Likewise, $\dfe{x,a}^-$ is strictly positive.
\end{proof}

\begin{prop} \label{p:b-fet-dx}
  Fix $t>0$ and $a>0$.  Then for $x\in (-a,a)$,
  \begin{align} \label{e:p-x}
    \dfe a^+(t,x)
    \le
    c_{\dfe{}} \min\{\dn_t(a-x) (a-x),\ \dn_t(a-|x|) (a+x)\},
  \end{align}
  where 
  \begin{align} \label{e:const-1}
    c_{\dfe{}}
    =
    2 t^{-1} \sumzi k \Sbr{\frac{(2ka+4a)^2} t + 1} e^{-2 k^2 a^2/t}.
  \end{align}
  Furthermore,
  \begin{align} \label{e:p-dx+}
    \partial_x \dfe a^+(t,a)
    :=
    \lim_{x\to a-} \partial_x \dfe a^+(t, x)
    =
    -\frac {\pi^2}{8 a^3}
    \sumoi k k^2 \exp\Cbr{-\frac{k^2\pi^2t}{8a^2}}\in (-\infty, 0)
  \end{align}
  and
  \begin{align} \label{e:p-dx-}
    \partial_x \dfe a^+(t,-a)
    :=\lim_{x\to (-a)+} \partial_x \dfe a^+(t, x)
    =
    \frac {\pi^2}{8 a^3}
    \sumoi k (-1)^{k-1} k^2 \exp\Cbr{-\frac{k^2\pi^2t}{8a^2}}\in
    (0,\infty).
  \end{align}
\end{prop}

\begin{proof}
  Recall $t \dfh y(t) = y \dn_t(y)$.   Let $y = a-x$.  Then by \eqref
  {e:b-fet-top1}, $\dfe a^+(t,x) = \dfh y(t) - t^{-1} g(y)$, where
  \begin{align*}
    g(y)
    =
    t\sum_{k\ge 2\text{ even}} [\dfh {2ka-y}(t) -
    \dfh {2ka+y}(t)].
  \end{align*}
  By $g(0)=0$, $g(y) = g'(\theta y) y$ for some $\theta = \theta(y)\in
  (0,1)$.  Then $\dfe a^+(t,x) \le t^{-1} y [\dn_t(y) + |g'(\theta
  y)|]$.   By $t\partial_y \dfh y(t) = -(y^2/t-1) \dn_t(y)$,
  \begin{align*}
    |g'(\theta y)|
    &\le
    \sum_{k\ge 2\text{ even}} \Cbr{
      \Sbr{\frac{(2ka-\theta y)^2} t + 1} \dn_t(2ka-\theta y)
      +
      \Sbr{\frac{(2ka+\theta y)^2} t + 1} \dn_t(2ka+\theta y)
    }.
  \end{align*}
  Since $y\in (0,2a)$, for $k\ge 2$, $2ka - \theta y \ge 2(k-2)a+y>0$
  and $2ka + \theta y \ge 2(k-1)a + y> 0$.  Also, for any $u\ge 0$,
  $\dn_t(u+y) \le e^{-u^2/(2t)} \dn_t(y)$.  Then
  \begin{align*}
    |g'(\theta y)|
    &\le
    \sum_{k\ge 2\text{ even}} \Cbr{
      \Sbr{\frac{(2ka)^2} t + 1}  \dn_t(2(k-2)a + y)
      +
      \Sbr{\frac{(2ka + 2a)^2} t + 1}
      \dn_t(2(k-1)a + y)
    }
    \\
    &\le
    \sum_{k\ge 2\text{ even}} \Cbr{
      \Sbr{\frac{(2ka)^2} t + 1} e^{-2(k-2)^2 a^2/t} 
      +
      \Sbr{\frac{(2ka+ 2a)^2} t + 1} e^{-2(k-1)^2 a^2/t} 
    }
    \dn_t(y).
  \end{align*}
  It follows that $\dfe a^+(t,x) \le c_{\dfe{}} y \dn_t(y) =
  c_{\dfe{}}(a-x) \dn_t(a-x)$.

  Let $z = a+x$.  Then by \eqref {e:b-fet-top1}, $\dfe a^+(t,x) =
  t^{-1} h(z)$, where
  \begin{align*}
    h(z) = 
    t\sum_{k\ge 1\text{ odd}} [\dfh {2ka - z}(t) - \dfh {2ka+z}(t)].
  \end{align*}
  Since $h(0)=0$, there is $\theta = \theta(z)\in (0,1)$ such that
  $h(z) = h'(\theta z) z$.  Similar to the above argument,
  \begin{align*}
    |h'(\theta z)|
    &\le
    \sum_{k\ge 1\text{ odd}}
    \Cbr{
      \Sbr{\frac{(2ka-\theta z)^2} t + 1} \dn_t(2ka-\theta z)
      +
      \Sbr{\frac{(2ka+\theta z)^2} t + 1} \dn_t(2ka+\theta z)
    } 
    \\
    &=
    \Sbr{\frac{(2a-\theta z)^2} t + 1} \dn_t(2a-\theta z)
    + I_1 + I_2
    \\
    &\le
    [(2a)^2/t+1] \dn_t(2a-z) + I_1 + I_2,
  \end{align*}
  where
  \begin{align*}
    I_1
    &=
    \sum_{k\ge 3\text{ odd}}
    \Sbr{\frac{(2ka-\theta z)^2} t + 1} \dn_t(2ka-\theta z)
    \le
    \sum_{k\ge 3\text{ odd}}
    \Sbr{\frac{(2ka)^2} t + 1} \dn_t(2(k-2)a+ z),
    \\
    I_2
    &=
    \sum_{k\ge 1\text{ odd}}
    \Sbr{\frac{(2ka+\theta z)^2} t + 1} \dn_t(2ka+\theta z)
    \le
    \sum_{k\ge 1\text{ odd}}
    \Sbr{\frac{(2ka+2a)^2} t + 1} \dn_t(2(k-1) a+ z).
  \end{align*}
  Then $\dfe a^+(a,x) \le t^{-1}[(2a)^2/t + 1]\dn_t(a-x) z +
  (c_{\dfe{}}/2)\dn_t(z) z\le c_{\dfe{}} \dn_t(a-|x|) (a+x)$.

  The equalities in \eqref{e:p-dx+} and \eqref{e:p-dx-} can be shown
  by direct calculation using \eqref {e:b-fet2} and \eqref
  {e:b-fet-top2}.  Then it is clear that $\partial_x \dfe a^+(t, a)$ is
  strictly negative and finite and $\partial_x \dfe a^+(t,-a)$ is
  finite.  To show that the latter partial derivative is strictly
  positive, first, write it as
  \begin{align*}
    \frac{\pi^2}{8 a^3} \sum_{k\ge 1\text{ odd}}[g_s(k^2) -
    g_s((k+1)^2)],
  \end{align*}
  where $g_s(x) = x \exp\{-s x\}$ with $s=\pi^2 t/8a^2$.  If $s\ge 1$,
  i.e., $t\ge 8 a^2/\pi^2$, then $g_s(x)$ is strictly decreasing on
  $[1,\infty)$ and hence $\partial_x \dfe a^+(t, -a)>0$.  To cover the
  case $0<t< 8 a^2/\pi^2$, differentiate \eqref {e:b-fet-top1} instead.
  Using the expressions in \eqref{e:b-fht},
  \begin{align*}
    \partial_x \dfe a^+(t, -a)
    =
    \frac{2}{\sqrt{2\pi} t^{3/2}} \sum_{k\ge 1 \text{ odd}} h_s(k^2),
  \end{align*}
  where $h_s(x) = (s x-1) e^{-s x/2}$ with $s = 4 a^2/t$.  It is seen
  that if $0<t< 4a^2$, in particular, if $0<t< 8 a^2/\pi^2$, then the
  sum on the right hand side is strictly positive.
\end{proof}

From the proof of \eqref {e:p-x} it is seen that the following is
actually true.
\begin{cor} \label{c:fet-dx}
  Fix $t>0$ and $a>0$.  Let $c_{\dfe{}}$ be as in \eqref{e:const-1}.
  Then for $x\in (-a,a)$ and $n\ge 0$,
  \begin{align*}
    \Abs{
      \sum^n_{k=0} [\dfh {4ka+a-x}(t) - \dfh {4ka+3a+x}(t)]
      + \dfh {4(n+1)a+a-x}(t)
    } \le c_{\dfe{}} \dn_t(a-x)(a-x)
  \end{align*}
  and
  \begin{align*}
    \Abs{
      \sum^n_{k=0} [\dfh {4ka+a-x}(t) - \dfh {4ka+3a+x}(t)]
    } \le c_{\dfe{}} \dn_t(a-|x|)(a+x).
  \end{align*}
\end{cor}

\begin{prop} \label{p:x}
  Given $a>0$,
  \begin{align*}
    \pr^0\{\bm t \in \dd x,  \fet a>t\}
    = 
    \cf{|x|<a} \dfex a(t,x)\,\dd x,
  \end{align*}
  where
  \begin{align} \label{e:pre-exit-BM}
    \dfex a(t,x)
    &=
    \sumii k (-1)^k \dn_t (x - 2k a)
    \\\label{e:pre-exit-BM-ft}
    &=
    \frac 1 a
    \sumzi k \exp\Cbr{-\frac{(2k+1)^2 \pi^2}{8 a^2} t}
    \cos\frac{(2k+1)\pi x}{2a}.
  \end{align}
  In addition, for all $x\in (-a,a)$, $\dfex a(t,x)>0$ and 
  \begin{align} \label{e:q-x}
    \dfex a(t,x) \le c_{\dfex{}} \dn_t(x) (a-|x|)
  \end{align}
  where $c_{\dfex{}} = 8 a t^{-1} \sumzi k (k+1) e^{-2k^2 a^2/t}$.
  Finally,
  \begin{align} \label{e:pre-exit-BM-edge}
    \dfex a(t,\pm a)=0, \qquad 
    \partial_x \dfex a(t, a) = -\partial_x q_a(t, -a) =
    -\dfe a(t,0)<0. 
  \end{align}
\end{prop}
\begin{proof}
  The assertion is trivial if $|x|\ge a$.  For $|x|<a$,
  \begin{align*}
    \pr^0\{\bm t\in \dd x, \fet a>t\}
    =
    \pr^0\{\bm t\in \dd x\}
    - 
    \pr^0\{\bm t\in \dd x, \fet a\le t\}.
  \end{align*}
  By the strong Markov property,
  \begin{align*}
    \pr^0\{\bm t\in \dd x, \fet a\le t, \bm{\fet a}=a\}
    &=
    \int^t_0 \pr^0\{\bm t\in \dd x\gv \fet a=s, \bm s=a\}
    \pr^0\{\fet a\in \dd s, \bm{\fet a}=a\}
    \\
    &=
    \frac 12\int^t_0 \pr^0\{\bm{t-s}\in \dd x - a\}
    \dfe a(s,0)\,\dd s.
  \end{align*}
  By \eqref{e:b-fet1} and the strong Markov property again,
  \begin{multline*}
    \pr^0\{\bm t\in \dd x, \fet a\le t, \bm{\fet a}=a\}
    =
    \sumoi k (-1)^{k-1}
    \int^t_0 \pr^0\{\bm{t-s}\in \dd x - a \} \dfh {(2k-1)a}(s)\, \dd s
    \\
    =
    \sumoi k (-1)^{k-1}
    \pr^0\{\bm t\in \dd x + 2(k-1) a, \fht {(2k-1) a}\le t\}
    =
    \sumoi k (-1)^{k-1}
    \pr^0\{\bm t\in 2 ka - \dd x\},
  \end{multline*}
  where the last equality is due to the reflection principle.  For
  each $k$,  $\pr^0\{\bm t\in 2 ka - \dd x\} = \dn_t(x - 2ka)\,\dd 
  x$.  A similar formula holds for $\pr^0\{\bm t\in \dd x, \fet a\le
  t, \bm{\fet a}=-a\}$.  Combining the results, \eqref {e:pre-exit-BM} 
  follows.  By \eqref {e:heat-eigen}, for any $g\in C([-a,a])$, 
  \begin{align*}
    \int^a_{-a} g(x) \pr^0(\bm t\in \dd x, t<\fet a)
    &=
    \mean^0[g(\bm t)\cf{t<\fet a}]
    \\
    &=
    \sum_{k\ge 1, k\text{ odd}}
    \frac 1 a
    \exp\Cbr{-\frac{k^2\pi^2 t}{8 a^2}} 
    \int^a_{-a} g(x) \cos\frac{k\pi x}{2a} \,\dd x.
  \end{align*}
  Comparing the measures on both sides yields \eqref
  {e:pre-exit-BM-ft}.

  To show $\dfex a(t,x)>0$ for $|x|<a$, fix $s\in (0,t)$ and $0<
  \delta < (a - |x|)/2$.  Then $(x-\delta,x+\delta)\subset (-a,a)$, so
  by the Markov property,
  \begin{align*}
    \pr^0\{\bm t \in \dd x,  \fet a>t\}
    &\ge
    \int^{x+\delta}_{x-\delta}
    \pr^0\{\bm s\in\dd u, \fet a>s\}
    \pr^u\{\bm{t-s}\in \dd x, \fet a>t-s\}
    \\
    &\ge
    \int^{x+\delta}_{x-\delta}
    \pr^0\{\bm s\in\dd u, \fet a>s\}
    \pr^0\{\bm{t-s}\in \dd x - u, \fet \delta>t-s\},
  \end{align*}
  where the second inequality is due to $[-\delta, \delta]\subset
  [-a-u, a-u]$ for all $u\in [x-\delta,x+\delta]$.  As a result,
  \begin{align*}
    \dfex a(t,x) \ge \int^{x+\delta}_{x-\delta}
    \dfex a(s,u) \dfex \delta(t-s, x-u)\,\dd u.
  \end{align*}
  If $\dfex a(t,x)=0$, then $\tilde{\dfex{}}(u) := \dfex a(s, u) \dfex
  \delta(t-s, x - u) = 0$ for all $u\in (x-\delta, x+\delta)$.
  However, by \eqref {e:pre-exit-BM}, $\tilde{\dfex{}}$ is analytic on
  $\Coms$.  This leads to $\tilde{\dfex{}}(u)\equiv 0$, implying
  either $\dfex a(s, \cdot)\equiv 0$ or $\dfex \delta(t-s,\cdot)\equiv
  0$, which is impossible, as $\dfex a(t,0) > 0$ by \eqref
  {e:pre-exit-BM-ft}.  The contradiction shows $\dfex a(t, x)>0$.
  
  To show \eqref{e:q-x}, since $\dfex a(t,x) = \dfex a(t,-x)$, it
  suffices to consider $x\ge 0$.  Let $z = a-x$.  Then $z\in (0,a)$
  and by \eqref {e:pre-exit-BM-ft}, 
  \begin{align*}
    \dfex a(t,x) = \sumzi k (-1)^k [\dn_t(2ka+a-z) -
    \dn_t(2ka+a+z)] := g(z).
  \end{align*}
  Since $g(0)=0$, there is $\theta = \theta(z)\in (0,1)$, such that
  $g(z) = g'(\theta z) z$.  By $\dn'_t(x) = -\dfh x(t)$
  \begin{align*}
    |g'(\theta z)|
    &\le
    2\sumzi k [\dfh {2ka+a-z}(t) + \dfh {2ka+a+z}(t)]
    \\
    &\le
    2 t^{-1} \sumzi k (2ka+a+z) [\dn_t(2ka+a-z) + \dn_t(2ka+a+z)]
    \\
    &\le
    8 a t^{-1}  \sumzi k (k+1) \dn_t(2ka+a-z)
    \\
    &\le
    8 a t^{-1} \sumzi k (k+1) e^{-2k^2a^2/t} \dn_t(a-z).
  \end{align*}
  Then \eqref{e:q-x} follows.

  It is straightforward to check the first equation in \eqref
  {e:pre-exit-BM-edge}.  By differentiating each term in \eqref
  {e:pre-exit-BM-ft} at $\pm a$ and comparing the resulting series to
  \eqref{e:b-fet2}, the second equation in \eqref {e:pre-exit-BM-edge}
  obtains.  The inequality in \eqref {e:pre-exit-BM-edge} is just the
  last assertion of Proposition \ref{p:b-fet}.
\end{proof}

\begin{cor}
  Given $a>0$ and $T>0$, for $t>0$ and $x\in (-a,a)$,
  \begin{align} \label{e:pre-exit-BM2}
    \pr^0\{\bm T\in \dd x, \fet a \in T + \dd t, \bm{\fet a} = \pm a\}
    =
    \dfex a(T,x) \dfe a^+(t, \pm x)\,\dd x\,\dd t.
  \end{align}
  As a result, under $\pr^0$, conditional on $\fet a = T+t$ and
  $\bm{\fet a}=a$, the \pdf~of $\bm T$ at $x\in (-a,a)$ is in
  proportion to $\dfex a(T,x) \dfe a^+(t, x)$ and conditional on $\fet
  a = T + t$ and $\bm{\fet a} = -a$, the \pdf~of $\bm T$ at $x\in
  (-a,a)$ is in proportion to $\dfex a(T,-x) \dfe a^+(t,-x)$.
\end{cor}
\begin{proof}
  Since $\{\fet a > T\} \in \Cal F(\bm s, s\le T)$, the Markov
  property of $\bm{}$ yields $\pr^0\{\bm T\in\dd x, \fet a \in T+ \dd
  t, \bm{\fet a} =a\} = \pr^0\{\bm T\in \dd x, \fet a>T\} \pr^x\{\fet
  a\in \dd t, \bm{\fet a}=a\}$.  Then the case $\bm{\fet a}=a$ of
  \eqref {e:pre-exit-BM2} follows from Propositions \ref{p:b-fet} and
  \ref{p:x}.  The case $\bm{\fet a}=-a$ is similarly proved.
  Finally, note $\dfex a (T,x) = \dfex a (T,-x)$.
\end{proof}

\section{Sampling of pre-exit location of a Brownian motion}
\label{s:sampling-pre-exit}

Fix $t>0$, $T>0$, and $a>0$.  The objective now is to sample $\bm t$
conditional on $\fet a = T + t$ and $\bm{\fet a}=\pm a$.  By symmetry,
it suffice to consider the case where $\bm{\fet a}=a$.  We 
shall construct an envelope function for $\dfe a^+(t,x)$ and one for
$\dfex a(T,x)$, respectively.  Consider $\dfe a^+(t,x)$ first.  To
emphasize that $x$ is the (only) variate, write $g_t(x) = \dfh x(t)$.  
Let
\begin{align*}
  P_k(x)
  =
  g_t(4ka+a-x) - g_t(4ka+3a+x).
\end{align*}
By \eqref {e:b-fet-top1}, for each $x\in (-a,a)$,
\begin{align} \label{e:p-term}
  \dfe a^+(t,x) = \sumzi k P_k(x).
\end{align}
Let
\begin{align} \label{e:p*}
  p^* = p^*(x,t)
  =
  \min\Cbr{
    n\ge 0: \sum^n_{k=0} P_k(x)\ge 0
  }.
\end{align}
Since $\dfe a^+(t,x)>0$ (Proposition \ref{p:b-fet}), the set on the
right hand side is nonempty, so $p^*$ is well defined.

\begin{prop} \label{p:p*}
  Fix $t>0$ and $a>0$.  Then for $x\in (-a,a)$, 1) $P_{p^*}(x)\ge 0$
  and $P_k(x)>0$ for all $k>p^*$, 2) as $x\to a$,
  \begin{align*}
    p^* \sim \sqrt{\frac{t}{8a^2}\ln \frac 1{a-x}},
  \end{align*}
  and 3) as $x\to -a$, $p^* = O(1)$.
\end{prop}
\begin{proof}
  1) By definition of $p^*$, for $x\in (-a,a)$, $P_{p^*}(x)\ge 0$, i.e.,
  $g_t(4p^*a+a-x) \ge g_t(4p^*a + 3 a + x)$.  Since $g_t(x)$ is
  strictly increasing on $(0, \sqrt t]$ and strictly decreasing on
  $[\sqrt t, \infty)$, $4p^*a+3a+x>\sqrt t$.  Then for $k>p^*$, $4ka +
  a - x > 4p^*a + 3a+x> \sqrt t$, so $P_k(x)>0$. 

  2) Put $y=a-x$.  By Corollary \ref{c:fet-dx}
  \begin{align} \label{e:p*-y}
    g_t(4p^*a+4a-y)
    \le
    \sum_{0\le k< p^*} P_k(a-y) + g_t(4p^*a+y) \le c_{\dfe{}}
    \dn_t(y) y.
  \end{align}
  On the other hand, by definition of $p^*$,
  \begin{align*}
    u(y):=\sum_{0\le k<p^*-1} P_k(a-y) +  g_t(4(p^*-1)a+y) 
    <
    g_t(4 p^*a-y).
  \end{align*}
  Since $u(0)=0$, there is $\theta = \theta(y)\in (0,1)$ such that
  $u(y) = u'(\theta y) y$.  As $y\to 0$, $p^* \toi$ by \eqref
  {e:p*-y}.  Then by Proposition \ref {p:b-fet-dx} and uniform
  convergence, $u'(\theta y) \to d: = - \partial_x \dfe a^+(t, a)>0$.
  As a result, $(1 + o(1)) d y \le g_t(4 p^* a -y)$.  This combined
  with \eqref {e:p*-y} yields the claimed asymptotic of $p^*$.

  3) Put $z = a+x$ and
  \begin{align*}
    v_n(z)
    =
    \frac 1 t \sum^n_{k=0} [g_t(4k a + 2 a - z) - g_t(4k a+ 2a + z)]
  \end{align*}
  if $n\ge 0$ and $v_n(z)=0$ if $n<0$.  Since $v_n(0)=0$, there is
  $\theta_n = \theta_n(z)\in (0,1)$ such that  $v_n(z) = v'_n(\theta_n
  z) z$.  It is not hard to show that as $z\to 0$, $\sup_n
  |v'_n(\theta_n z) - d_n|\to 0$, where $d_n = -2 t^{-1} \sum_{0\le
    k<n} g'_t(4ka+2a)$.  Since $d_\infty = \partial_x \dfe a^+(t,-a)$,
  which by Proposition \ref {p:b-fet-dx} is strictly positive, there is
  $n^*$ such that $d_n>0$ for all $n\ge n^*$.  Since $v_{p^*-1}(z)\le
  0$, this implies $p^*-1< n^*$ for all $z>0$ small enough.  Thus
  $p^*=O(1)$.
\end{proof}

The asymptotics of $p^*$ in Proposition \ref{p:p*} suggest that for
rejection sampling involving $\dfe a^+(t,x)$, $x\approx a$ should be
handled more carefully than $x\approx -a$.  Letting $y=a-x$, for $k\ge
1$,
\begin{align*}
  P_{p^*+k}(x)
  &\le
  g_t(4 (p^*+k) a + y)
  \\
  &\le
  g_t(4(p^*+1)a + y)
  \times \frac{4(p^*+k) a + y}{4(p^*+1) a + y} 
  e^{-8(k-1)^2a^2/t},
\end{align*}
where the second inequality is due to $\dn_t(y+z) \le \dn_t(y)
e^{-z^2 / (2t)}$ for all $y$, $z>0$.  From the proof of Proposition
\ref {p:p*}, $4(p^*+1)a-y\ge \sqrt t$ and $\dn(x)$ is strictly
decreasing on $[\sqrt t,\infty)$.  Thus $g_t(4(p^*+1)a + y) <
g_t(4(p^*+1)a - y)$ and by \eqref{e:p*-y}, 
\begin{align} \label{e:p-term1}
  0< P_{p^*+k}(x)
  \le
  c_{\dfe{}}\dn_t(a-x) (a-x) \times k e^{-8(k-1)^2 a^2/t}
\end{align}
and meanwhile
\begin{align} \label{e:p-term1x}
  0\le \sum^{p^*}_{k=0} P_k(x) \le c_{\dfe{}} \dn_t(a-x)(a-x).
\end{align}
We will use \eqref{e:p-term}--\eqref{e:p-term1x} to construct an
envelope for $\dfe a^+(t,x)$.

To construct an envelope for $\dfex a(T,x)$, for each $x\in (-a,a)$,
the series in \eqref {e:pre-exit-BM} converges absolutely.  Let
\begin{align*}
  Q_k(x)
  =
  \dn_T(4k a + |x|) - \dn_T(4k a + 2a - |x|) -
  \dn_T(4k a + 2 a + |x|) + \dn_T(4k a + 4a-|x|).
\end{align*}
Noting $\dfex a(T,x) = \dfex a(T,-x)$,
\begin{align} \label{e:q-term}
  \dfex a(T,x)=\sumzi k Q_k(x).
\end{align}
Let 
\begin{align} \label{e:q*}
  q^*=q^*(x,T)
  =
  \min\Cbr{n\ge 0: \sum^n_{k=0} Q_k(x)\ge 0}.
\end{align}
Since $\dfex a(T,x)>0$ by Proposition \ref{p:x}, $q^*$ is well
defined.
\begin{prop} \label{p:q*}
  Fix $T>0$ and $a>0$.  Then for $x\in (-a,a)$, 1) $Q_{q^*}(x)\ge 0$,
  $Q_k(x)>0$ for all $k>q^*$, and 2) as $|x|\to a$, $q^* = O(1)$.
\end{prop}
Indeed, since $\dn_T$ is strictly concave on $(0,\sqrt T]$ and
strictly convex on $[\sqrt T, \infty)$, 1) follows by similar argument
for 1) of Proposition \ref{p:p*}.  On the other hand, 2) follows from
\eqref{e:pre-exit-BM-edge} and similar argument for 3) of Proposition
\ref{p:p*}.  The detail of the proof is omitted for brevity.

Since $\dn_T$ is strictly decreasing on $[0,\infty)$, for each $k\ge
0$, $Q_k(x)  < \dn_T(4k a + |x|) - \dn_T(4k a + 2a-|x|) = -2\dn'_T(4k a
+ y) (a-|x|)$, where $y = y(x,k)\in (|x|, 2a-|x|)$.  By
$|\dn'_T(4ka+y)| \le T^{-1} (4ka+2a) \dn_T(4ka+|x|) \le T^{-1}
(4ka+2a) e^{-8 k^2 a^2/T}\dn_T(x)$,
\begin{gather} \label{e:q-term1}
  Q_k(x) \le 4 a T^{-1}(2k+1) e^{-8 k^2 a^2/T} \dn_T(x)(a-|x|), \quad
  k>q^*.
\end{gather}
Meanwhile,
\begin{align} \label{e:q-term2}
  0\le \sum^{q^*}_{k=0} Q_k(x)
  <
  4 a T^{-1}(2q^*+1) e^{-8 (q^*)^2 a^2/T} \dn_T(x)(a-|x|).
\end{align}

Now \eqref {e:p-term}--\eqref {e:q-term2} can be combined as follows.
Define
\begin{gather} \label{e:gamma-t-T}
  \gamma(x) = \gamma(x, T,t)
  =
  4 a T^{-1} \cf{|x|<a}\dn_t(a-x)(a-x) \dn_T(x) (a-|x|).
\end{gather}
For $k\ge 0$, define
\begin{align} \label{e:a-k-t}
  a_k
  &=
  a_k(t)
  =
  \begin{cases}
    1 & k=0 \\
    2k e^{-8(k-1)^2 a^2/t} & k\ge 1
  \end{cases}
\end{align}
and
\begin{align} \label{e:b-k-T}
  b_k = b_k(T) = (2 k + 1) e^{-8 k^2 a^2/T}.
\end{align}
For $x\in (-a,a)$ and $k\ge 0$, define
\begin{align} \label{e:rk}
  r_k(x, m)
  =
  \begin{cases}
    \displaystyle
    \frac{\sum^m_{k=0} P_k(x)}{c_{\dfe{}} \dn_t(a-x) (a-x) a_k}
    &
    \text{if } k=0
    \\[2ex]
    \displaystyle
    \frac{P_{m+k}(x)}{c_{\dfe{}} \dn_t(a-x)(a-x) a_k}
    &
    \text{else} 
  \end{cases}
\end{align}
and
\begin{align} \label{e:sk}
  s_k(x,m)
  =
  \begin{cases}
    0 & \text{if } k<m
    \\
    \displaystyle
    \frac{\sum^m_{k=0} Q_k(x)}{4 a T^{-1} \dn_T(x)(a-|x|) b_k}
    & \text{if } k=m
    \\[2ex]
    \displaystyle
    \frac{Q_k(x)}{4 a T^{-1} \dn_T(x)(a-|x|) b_k} &
    \text{else}
  \end{cases}
\end{align}
Then $r_k(x, p^*)\in [0,1)$, $s_k(x, q^*)\in [0,1)$ for all $k\ge 0$
and
\begin{align*}
  \dfe a^+(t, x)\dfex a(T,x)
  =
  \gamma(x)
  \times
  \sumzi k a_k r_k(x, p^*)
  \times
  \sumzi k b_k s_k(x, q^*),
\end{align*}
The rejection sampling of the \pdf\ $\gamma(x)/\int^a_{-a} \gamma$ is
quite routine, although for efficiency, the detail has to depend on
$a$, $t$, and $T$ (more precisely, on $a/\sqrt t$ and $a/\sqrt T$).
On the other hand, note that both $a_k$ and $b_k$ are log-concave,
i.e., $a_{k-1} a_{k+1} \ge a^2_k$ and $b_{k-1} b_{k+1} \ge b^2_k$.  It
was shown in \cite {devroye:87:comp} that a log-concave distribution
on integers can be sampled efficiently by rejection sampling.  The
only minor issue here is that the values of the normalizing constants
for $a_k$ and $b_k$ in general are not available exactly.  However, it
is easy to find positive lower and upper bounds for the sequences.
When these bounds are used in place of the exact normalizing
constants, the rejection sampling in \cite {devroye:87:comp} still
works with some minor modifications and loss of efficiency.

\begin{algorithm}[ht]
  \caption{Sampling $\bm T$ under $\pr^0$, conditional on $\fet a =
    T+t$ and $\bm{\fet a}=a$}
  \label{a:bm-x}
  \begin{algorithmic}[1]
    \REQUIRE $a\in (0,\infty)$, $T\in (0,\infty)$, $t\in (0,\infty)$
    \REPEAT 
    \STATE Sample $X\sim \gamma/\int \gamma$, where $\gamma(x) =
    \gamma(x,T,t)$ is defined in \eqref{e:gamma-t-T}
    \STATE Compute $p^*=p^*(X,t)$ by \eqref{e:p*} and $q^* = q^*(X,T)$ by
    \eqref{e:q*}
    \STATE Sample $\kappa_1\in \{0,1,\ldots\}$ with \pmf\ $a_k/\sum_j
    a_j$, where $a_k = a_k(t)$ is defined in \eqref{e:a-k-t} 
    \STATE Sample $\kappa_2\in \{0,1,\ldots\}$ with \pmf\ $b_k/\sum_j
    b_j$, where $b_k = b_k(T)$ is defined in \eqref{e:b-k-T}
    \STATE Sample $U\sim \dunif(0,1)$
    \UNTIL $r_{\kappa_1}(X, p^*) s_{\kappa_2}(X, q^*) \le U$,
    where $r_k$ and $s_k$ are defined in \eqref{e:rk} and
    \eqref{e:sk}, respectively
    \RETURN $X$
  \end{algorithmic}
\end{algorithm}

The above results lead to the rejection sampling Algorithm
\ref{a:bm-x}.

\section{Comments} \label{s:comments}
This paper only considers \levy processes that can be embedded into a
subordinated standard Brownian motion.  In principle, the scheme
illustrated in Figure \ref{f:fet} can be applied to any \levy process,
for example, a spectrally negative $\alpha$-stable \levy process with
$\alpha \in [1,2)$.  Indeed, without subordination being involved, the
scheme can be somewhat simplified.  However, even for the spectrally
negative $\alpha$-stable process, which has many remarkable
properties, there are few close form formulas for its first exit event
(cf.\ \cite {bertoin:96:cup, doney:07:sg-b, kyprianou:06:sv-b}).
Since all \levy processes are differences of independent spectrally
negative \levy processes, it would be interesting to find exact
sampling methods for the first exit event of the latter.

The procedure in Algorithm \ref{a:fet} in principle may also be
applied to a subordinated Brownian motion that has a drift or takes
values in $\Reals^d$ with $d>1$.  For example, in each iteration,
intead of an interval centered at the current value of a Brownian
motion (cf.\ step \ref{a:fet-a}), use a sphere of small radius so that
jumps of large size can be detected and removed.  One potential
problem is that for $d>1$, when there is a positive chance for the
\levy process to creep across the boundary of a region, the iteration
in Algorithm \ref{a:fet} may not be able to stop.  This is because the
chance for the Brownian motion to hit the ``right spots'' on the
sphere, namely, the intersection between the sphere and the boundary
of the region, can be 0; at any other spot on the sphere, a new sphere
has to be constructed, possibly with a smaller radius.  As a result,
the iteration will approach infinitely but never reach the location of
the first exit.  Regardless, there is no extra work on the sampling
for the subordinator.  Also, the first exit time out of a sphere by a
Brownian motion is well understood.  On the other hand, the
distribution of the pre-exit value of the Brownian motion becomes
substantially subtle.  Provided the distribution is available in
close form, the procedure of the paper can be extended
straightforwardly.

\begin{hide}
\def\dc{\nu}            
For $\dc\ne0$, denote
\begin{align*}
  \fht a\Sp\dc = \inf\{t: \dc t + \bm t = a\}
\end{align*}
and
\begin{align*}
  \fet a\Sp\dc
  =
  \inf\{t: \dc t + \bm t\in\{-a, a\} \}
  =
  \fht{-a}\Sp\dc\wedge \fht a\Sp\dc. 
\end{align*}
It is known that
\begin{align}  \label{e:b-fht-drift}
  \pr^x\{\fht a\Sp\dc \in \dd t\}
  = \frac{|a-x|}{\sqrt{2\pi} t^{3/2}} \exp\Sbr{
    -\frac{(a-x-\dc t)^2}{2t}
  }\,\dd t
\end{align}
and
\begin{align} \label{e:b-fht-drift2}
  \pr^x\{\fht a\Sp\dc = \infty\}
  =
  \begin{cases}
    0 & \text{if } \dc(a-x)\ge 0 
    \\
    1- e^{2\dc(a-x)} & \text{else}
  \end{cases}
\end{align}
(\cite{borodin:02:bvb}, p.~295, 2.0.2), while for $a>0$ and $x\in
(-a,a)$,
\begin{align} \label{e:b-fet-top1-drfit}
  \begin{split}
    \pr^x\{\fet a\Sp\dc = \fht{-a}\Sp\dc\in\dd t\}
    &=
    e^{-\dc(x+a) - \dc^2 t/2} \dfe a^-(t,x)\,\dd t,
    \\
    \pr^x\{\fet a\Sp\dc = \fht a\Sp\dc\in\dd t\}
    &=
    e^{-\dc(x-a) - \dc^2 t/2} \dfe a^+(t,x)\,\dd t
  \end{split}
\end{align}
(\cite{borodin:02:bvb}, p.~309, 3.0.6).
\end{hide}

\begin{small}

\end{small}

\end{document}